\documentclass[11pt]{amsart}
\usepackage{amsmath}
\usepackage{amssymb}
\usepackage{amsthm}
\usepackage{amsfonts}
\usepackage{array}
\usepackage{xy}
\usepackage[pdftex]{graphicx}
\usepackage{verbatim} 
\usepackage{ textcomp }
\usepackage{yfonts}
\usepackage{tikz}
\newtheorem{theorem}{Theorem}[section]
\newtheorem{proposition}[theorem]{Proposition}
\newtheorem{lemma}[theorem]{Lemma}
\newtheorem{corollary}[theorem]{Corollary}
\theoremstyle{definition}

\newtheorem{example}{Example}
\theoremstyle{remark}
\newtheorem{remark}[theorem]{Remark}

\newcommand{\C}{\mathbb{C}}

\newcommand{\Z}{\mathbb{Z}}
\newcommand{\N}{\mathbb{N}}

\newcommand{\R}{\mathbb{R}}

\newcommand{\T}{\mathbb{T}}

\begin{document}

\title{Almost strong mixing group actions in topological dynamics}

\author{Jane Wang}

\maketitle 

\begin{abstract}  In ergodic theory, given sufficient conditions on the system, every weak mixing $\mathbb{N}$-action is strong mixing along a density one subset of $\mathbb{N}$. We ask if a similar statement holds in topological dynamics with density one replaced with thickness. We show that given sufficient initial conditions, a group action in topological dynamics is strong mixing on a thick subset of the group if and only if the system is $k$-transitive for all $k$, and conclude that an analogue of this statement from ergodic theory holds in topological dynamics when dealing with abelian groups. 
\end{abstract}

\section{Background and Problem Statement} 
\label{sec:prelim}

Ergodic theory and topological dynamics are two branches of the larger subject of dynamical systems. The basic setup of a dynamical system in ergodic theory is to have $T$ a measure-preserving transformation acting on a probability space $(X, \mathcal{B}, \mu)$. If we want to work with more generality, we can replace the transformation $T$ with $\Gamma$, a measurable group action on $X$. 

In topological dynamics, we no longer require our underlying space to be a measure space, but instead to be a topological space. In this setup, our topological system is $(X, \Gamma)$, where $X$ is a compact, usually metrizable, space\footnote{The restriction that $X$ is compact and metrizable mirrors the restriction that $X$ is a probability space in the setting of ergodic theory. However, for the purposes of this paper, we will usually work in more generality and only assume that $X$ is second countable.} and $\Gamma$ a topological group acting continuously on $X$, so that the map $X \times \Gamma \rightarrow X$ is continuous. We will always implicitly make the assumption that both $X$ and $\Gamma$ are Hausdorff. Furthermore, we will generally take $\Gamma$ to also be locally compact, and second countable. 

Although ergodic theory and topological dynamics are different fields with their own theories, many concepts that are used to describe dynamical systems make an appearance in both fields (e.g. recurrence, periodicity, mixing, chaos). In addition, many of the results involving these ideas hold in both fields  (see \cite{GW}). 

For example, various notions of mixing exist in both the measurable and topological settings. In both ergodic theory and topological dynamics, there are notions of strong and weak mixing and in both theories, a well-known result is that strong mixing implies weak mixing. 

In ergodic theory, it is also known that given sufficient conditions on the system, a weak mixing system is ``almost" strong mixing, for some notion of ``almost" that we will make more precise in the next section. We can then ask if a similar statement can be made in the context of topological dynamics. This then leads us to the motivating question of this paper: when is a topological weak mixing system ``almost" strong mixing?

\subsection{Motivation from Ergodic Theory}

We begin by describing the known results in ergodic theory. In the setting of measurable dynamics, we say that a measure-preserving transformation $T$ on a probability space $(X, \mathcal{B}, \mu)$  is \textbf{weak mixing} if 
$$\lim_{n \rightarrow \infty} \frac{1}{n} \sum_{k=0}^{n-1} |\mu(TA^{-k} \cap  B) - \mu(A) \mu (B) | = 0$$ for all $A, B \in \mathcal{B}$. Furthermore, we say that a system is \textbf{strong mixing} if $$\lim_{n \rightarrow \infty} \mu(TA^{-n} \cap B) = \mu(A) \mu (B).$$
We can see that strong mixing implies weak mixing, but the following two theorems will give us that systems that are weak mixing ``look" strong mixing on a density one subset of $\N$, where $J \subset \N$ is said to be \textbf{density one} if $$\lim_{n \rightarrow \infty} \frac{\#\{J \cap [1,n]\}}{n} = 1.$$

These theorems and their proofs may be found in \cite{W}, Section 1.7. 

\begin{theorem} 
\label{thm:mdo1}
Let $T$ be a measure-preserving transformation on a probability space $(X, \mathcal{B}, \mu)$ that is weak mixing. Then, for any $A, B \in \mathcal{B}$, there exists a density one set $J \subset \N$ for which $$\lim_{n \in J, n \rightarrow \infty} \mu(T^{-n} A \cap B) = \mu(A) \mu(B).$$ 
\end{theorem}
We can interpret this theorem as saying that $T$ is strong mixing along $J$. We notice here that in the above theorem, the density one set $J$ may depend on $A$ and $B$. However, under the extra condition that there exists countable basis for $(X, \mathcal{B}, \mu)$, we can remove the dependence on the sets $A$ and $B$: 

\begin{theorem}
\label{thm:mdo2}
Let $T$ be a measure-preserving transformation on a probability space $(X, \mathcal{B}, \mu)$ with a countable basis. Then, $T$ is weak mixing if and only if there exists of density one set $J \subset \N$ such that for all $A, B \in \mathcal{B}$, $$\lim_{n \in J, n \rightarrow \infty} \mu(T^{-n} A \cap B) = \mu(A) \mu(B).$$
\end{theorem}

Our goal is now to find a similar result in the context of topological dynamics. We will formalize the statement of the problem that we will be looking at in the next section.

\subsection{The Picture in Topological Dynamics} 

In topological dynamics, we generally begin with $X$ a compact, metrizable space and $\Gamma$ a locally compact, second countable topological group that acts on $X$. We say then that $(X, \Gamma)$ is \textbf{(topologically) weak mixing} if for any choice of four non-empty open sets $U_1, U_2, V_1, V_2 \subset X$, there exists a $\gamma \in \Gamma$ for which 
$$\gamma U_1 \cap V_1 \neq \emptyset \text{ and } \gamma U_2 \cap V_2 \neq \emptyset.$$
Furthermore, we say that $(X, \Gamma)$ is \textbf{(topologically) strong mixing} if for any two non-empty open sets $U, V \subset  X$, there exists a compact subset $F \subset \Gamma$ for which 
$$\gamma U \cap V \neq \emptyset \text{ for all } \gamma \in \Gamma \backslash F.$$ 
As in the measurable setting, we can see that if $\Gamma$ is not compact, then topological strong mixing implies topological weak mixing. If $\Gamma$ is compact, then $(X, \Gamma)$ is trivially strong mixing. Therefore, in this paper we will only consider the actions of noncompact groups $\Gamma$.

To find an analogues of Theorems \ref{thm:mdo1} and \ref{thm:mdo2}, we need a notion similar to density one in the context of general groups. Our candidate here is the notion of thickness. 

We say that a set $N \subset \Gamma$ is \textbf{thick} if for every finite subset $F \subset \Gamma$, there exists a $\gamma_0 \in N$ with $F^{-1} \gamma_0 \subset N$, which is equivalent to saying that $$N \cap \bigcap_{\gamma \in F} \gamma N \neq \emptyset.$$ As we will be often concerned with the set of $\gamma$'s that send one open set to another, we introduce the following notation.  For a topological system $(X, \Gamma)$ and sets $A, B \subset X$, we let $$N(A, B) := \{ \gamma \in \Gamma : \gamma A \cap B \neq \emptyset\}.$$ For $\Gamma$ an abelian, discrete, countable group, a result like Theorem \ref{thm:mdo1} exists in the topological setting. 

\begin{theorem}[\cite{G}, Theorem 1.11]
\label{thm:tdt}
For $\Gamma$ an abelian, discrete, countable group, and $X$ a compact space, the following conditions are equivalent: 
\begin{itemize}
\itemsep0em 
\item[(i)] $(X, \Gamma)$ is weakly mixing. 
\item[(ii)] For every four open sets $U_1, U_2, V_1, V_2 \subset X$ such that $N(U_1, V_1) \neq \emptyset$ and $N(U_2, V_2) \neq \emptyset$, there exist nonempty open sets $U, V$ for which $N(U,V) \subset N(U_1, V_1) \cap N(U_2, V_2)$. 
\item[(iii)] For every $k \in \N$, $(X, \Gamma)$ is $k$ transitive. That is, $(X^k, \Gamma)$ is topologically transitive. 
\item[(iv)] For every part $U, V$ of nonempty open sets, the set $N(U,V)$ is thick. 
\item[(v)] For every pair $U, V$ of nonempty open sets, there exists a $\gamma \in \Gamma$ for which $\gamma U \cap U \neq \emptyset$ and $\gamma U \cap V \neq \emptyset$. 
\end{itemize}
\end{theorem}

That is, the implication $(i)$ to $(iv)$ gives us that $(X, \Gamma)$ being weak mixing implies that for any two open sets $U, V \subset X$, there exists the thick set $N(U,V)$ on which we have strong mixing. One goal that we will have is to extend this characterization to larger classes of groups. 

We can hope also for an analogue of Theorem \ref{thm:mdo2} in the topological setting. For a topological system $(X, \Gamma)$, we define the notion of thick strong mixing. That is, we say that $(X, \Gamma)$ is \textbf{thick strong mixing} if there exists a thick subset $N \subset \Gamma$ on which $(X, N)$ is mixing. We note here that saying that $(X, N)$ is mixing is a slight abuse of notation as $N$ may not be a group. However, we understand this to mean that for every pair of nonempty open sets $U, V \subset X$, there exists a compact set $F \subset N$ such that $\gamma U \cap V \neq \emptyset$ for all $\gamma \in N \backslash F$. 

Now, the problem that we will investigate becomes the following: 

\textbf{Problem Statement:} For what conditions on $X$ and $\Gamma$ does weak mixing of the system $(X, \Gamma)$ imply that it is also thick strong mixing? 

The main result of this paper will be the equivalence of $k$-transitivity for all $k$ and thick strong mixing. We state this result here, with $(iii)$ being an intermediate result that we use in our proofs. 

\begin{theorem}
\label{thm:main}
Let $(X, \Gamma)$ be a topological system where $X$ is second countable space and $\Gamma$ is a second countable, locally compact topological group that is not compact. Then, the following are equivalent: 
\begin{itemize}
\itemsep0em 
\item[(i)] $(X, \Gamma)$ is $k$-transitive for all $k \in \N$. That is, $(X^k, \Gamma)$ is topologically transitive. 
\item[(ii)] $(X, \Gamma)$ is thick strong mixing. That is, there exists a thick subset $N \subset \Gamma$ on which $(X, N)$ is mixing. 
\item[(iii)] For any set of nonempty open sets $\{U_1, \ldots, U_n, V_1, \ldots, V_n\}$ in $X$, $\bigcap_{1 \leq i \leq n} N(U_i, V_i)$ is thick. 
\end{itemize}
\end{theorem}

The outline of the remainder of this paper is as follows. We will begin by examining one particular weak mixing system, the Chac\'{o}n System $(X, \Z)$, and constructing the thick subset of $\Z$ on which it is strong mixing. We will then use ideas from this example to prove the equivalence of $k$-transitivity for all $k$ and thick strong mixing on a large family of topological systems. As all abelian weak mixing systems are $k$-transitive for all $k$, this will give that for abelian group actions, weak mixing implies thick strong mixing. Finally, we will look at the nonabelian case. We will show that weak mixing does not always imply thick strong mixing for nonabelian groups, and will give a characterization of when an analogue of Theorem \ref{thm:mdo1} holds in topological dynamics.

\section{A Motivating Example}
\label{sec:chacon}

Before we begin classifying for which topological systems weak mixing implies thick strong mixing, it is first good to know that there exist systems that are not strong mixing, but are weak mixing and thick strong mixing. In particular, we will look at the Chac\'{o}n system. In proving that the Chac\'{o}n system is thick strong mixing, we will use some techniques that we will build on in our proof of the equivalence of $k$-transitivity for all $k$ and thick strong mixing. 

\subsection{Introduction to Shift Systems}

We start with a quick overview of shift systems. Let $\mathcal{A} = \{0, 1, \ldots, a-1\}$  and $\Omega(\mathcal{A}) = \mathcal{A}^\Z$ be the \textbf{full shift space} on $a$ letters, the space of all bi-infinite sequences on the alphabet $\mathcal{A}$. As we are thinking of elements in $\Omega(\mathcal{A})$ as strings, we will use the notation $a_0a_1\cdots a_n$ to denote the substring consisting of the letters $a_0, a_1, \ldots, a_n$. Then, a basis for the topology on this space is formed by all \textbf{cylinder sets} $[a_0a_1\cdots a_m]_k$, consisting of all $\omega = (\cdots \omega_{-2}\omega_{-1} \omega_0 \omega_1 \omega_2 \cdots) \in \Omega(\mathcal{A})$ such that $\omega_{k+i} = a_i$ for all $0 \leq i \leq n$. Given this topology, we see that $\Omega(\mathcal{A})$ is a compact topological space. 

We can then define a transformation $S$ on this space of sequences such that $S$ acts on an $\omega$ in the space by left shift. That is, we have that $S\omega = \omega^\prime$ where $\omega^\prime_{i} = \omega_{i+1}$ for all $i \in \Z$. Thinking of shifts as as a $\Z$ action, we have then defined the shift system $(\Omega(\mathcal{A}), \Z)$. 

If instead of working with the whole space $\Omega(\mathcal{A})$, we work with some closed shift-invariant subspace $X$, we can also define the \textbf{subshift system} $(X, \Z)$. In the next section, we will explicitly define one such subshift system, the Chac\'{o}n system.

\subsection{The Chac\'{o}n System}

To construct the Chac\'{o}n system, we inductively define blocks $B_n$ by: $$B_1 = 0010, B_2 = 0010001010010, \ldots, B_{n+1} = B_n B_n 1 B_n.$$ We can think of this as beginning with $B_1 = 0010$ and in each stage, defining $B_{n}$ from $B_{n-1}$ by making the substitutions $0 \mapsto 0010$ and $1 \mapsto 1$. In this way, the system we will construct will also be a \textbf{substitution system}.  

Now, we define an $\omega \in \{0,1\}^\Z$ by for each $n$, setting $$\omega_{-l_n} \omega_{- l_n + 1} \cdots \omega_0 \omega_1 \ldots \omega_{l_n - 1} = B_n B_n,$$ where $l_n$ is the length of $B_n$. Then, we define a left-shift transformation $S$ on $w \in \{0,1\}^\Z$ by $Sw = w^\prime$ where $w^\prime_n = w_{n+1}$. Defining $X := \overline{\mathcal{O}}_S \omega$, the closure of the set of all (left and right) shifts of $\omega$, under the topology where the basis elements are the cylinder sets in $\{0,1\}^\Z$, we define the \textbf{Chac\'{o}n System} as $(X, \Z)$. 

While it is known that the Chac\'{o}n system is weak mixing but not strong mixing, we give proofs of these facts here to familiarize ourselves with this important example. We will also show that the Chac\'{o}n system is thick strong mixing. To see this, we first do some preliminary investigation into the intersections of cylinder sets in the Chac\'{o}n system: 

For any string in $A = a_0a_1\cdots a_n \in \{0,1\}^{n+1}$, where $n \in \N$, the cylinder set $[A]_k$ is the set of all $\omega = (\cdots \omega_{-2}\omega_{-1}\omega_0 \omega_1 \omega_2 \cdots) \in X$ for which $\omega_{k+i} = a_i$ for all $0 \leq i \leq n$. For ease of notation, we will use $[A]_0$ and $[A]$ interchangeably. 

For $[A]_k$ to be nonempty, there must be a place in $\omega$ where the string $a_0a_1\cdots a_n$ occurs, implying that $A$ is a substring of some block. $B_\ell$. Using the notation $m[A]_k$ to mean $S^m [A]_k$, we then have the containment $[A]_k \subset [B_\ell]_{\tilde k} = (-\tilde k) [B_\ell] \subset (-\tilde k) [B_{n}]$ for some $\tilde k$ and all $n \geq \ell$. 

Then, for any $a,c \in \Z$, if we look at the set $$N([A]_a, [C]_c) = \{ n : n[A]_a \cap [C]_c \neq \emptyset\},$$ we can find $k,m \in \Z$ such that $$N([A]_a, [C]_c) \supset N([B_k], m[B_k]).$$ It will turn out that to determine the mixing properties of the Chac\'{o}n system, it will be enough to look at these sets $N([B_k], m[B_k])$. 

We first give a description of the set $N([B_1], [B_1])$. To see what this set is, we notice that each occurrence of $B_1$ is contained in some string $B_2 = B_1B_1 1 B_1$. Therefore, $n[B_1] \cap [B_1] \neq \emptyset$ for $n \in H_1 = \{-9, -5, -4, 0, 4, 5, 9\}$, the set of all shifts that would allow a $B_1$ in $B_2$ to intersect another copy of $B_1$ in the same $B_2$.  

Now, we can consider copies of $B_2$ in $B_3 = B_2 B_2 1 B_2$. The length of $B_2$ is $13$, so $B_2$ can intersect itself in a copy of $B_3$ by translations by $n \in H_2 = \{-27, -14, -13, 0, 13, 14, 27\}$, and therefore $B_1$ may intersect another copy of $B_1$ by any translation in $H_2$ followed by any translation in $H_1$. For an arbitrary $m \in \N$, we denote by $H_m$ the set of translations that will map a copy of $B_m$ in $B_{m+1}$ onto itself, and as we can check that the length of $B_m$ is $\frac{3^{m+1}-1}{2}$, we have that 
$$H_m = \left\{- 3^{m+1}, - \frac{3^{m+1}+1 }{2}, - \frac{3^{m+1} - 1}{2} , 0, \frac{3^{m+1}-1}{2}, \frac{3^{m+1} +1}{2}, 3^{m+1}\right\}.$$
Hence, all $n$ for which $n[B_1] \cap [B_1] \neq \emptyset$ is $$N([B_1], [B_1]) = H_1 \oplus H_2 \oplus H_3 \oplus \cdots.$$ Similarly, the $n$ for which $n[B_k] \cap [B_k] \neq \emptyset$ is  $$N([B_k], [B_k]) = H_k \oplus H_{k+1} \oplus H_{k+2} \oplus \cdots. $$

Furthermore, we can see that $N([B_k], m[B_k]) = mN([B_k], [B_k])$. Now, we can prove the following concerning the mixing properties of the system. 

\begin{proposition}
The Chac\'{o}n System is weak mixing. 
\end{proposition}
\begin{proof}
We suppose that $U_1, U_2, V_1, V_2 \subset X$ are non-empty and open. Then, we can find nonempty cylinder sets $[\tilde U_1], [\tilde U_2], [\tilde V_1], [\tilde V_2]$, where $[\tilde U_1] \subset U_1$ and $\tilde U_1$ is a finite string of zeros and ones, and similarly for $U_2, V_1,$ and $V_2$. Now, we can find a $k$ large enough such that $\tilde U_1, \tilde U_2, \tilde V_1, \tilde V_2$ are all substrings in $B_k$, and $n,m,k_1, k_2 \in \Z$ such that 
\begin{align*} 
N(U_1, V_1) \cap N(U_2, V_2) & \supset N([\tilde U_1], [\tilde V_1]) \cap N([\tilde U_2], [\tilde V_2]) \\
& \supset N([B_k]_{k_1}, [B_k]_{k_1+m}) \cap N([B_k]_{k_2}, [B_k]_{k_2+n}) \\
& \supset N([B_k], m[B_k]) \cap N([B_k], n[B_k]) \\
& = m N([B_k], [B_k]) \cap n N([B_k], [B_k]).
\end{align*} 

We may assume without loss of generality that $m > n$. Since $N([B_k], [B_k]) = H_k \oplus H_{k+1} \oplus H_{k+2} \oplus \ldots,$ we see that both $$a = \sum_{i=k}^{k + (m-n) - 1} \frac{3^{i+1} - 1}{2} \text{ and } b = \sum_{i=k}^{k + (m-n)-1} \frac{3^{i+1} +1}{2}$$ are contained in $N([B_k], [B_k])$. But $b = a + (m-n)$, so $$m+a = b+n \in m N([B_k], [B_k]) \cap n N([B_k], [B_k]).$$
This implies then that $N(U_1, V_1) \cap N(U_2, V_2) \neq \emptyset$ and the Chac\'{o}n system is weak mixing. 
\end{proof} 

We can also show that the Chac\'{o}n System is not strong mixing. 

\begin{proposition} The Chac\'{o}n System is not strong mixing. 
\end{proposition}
\begin{proof} We consider the set $$G = \left\{ \frac{3^{m+1}-3}{2} : m \in \N\right\}$$ and claim that no $g \in G$ is in $H_1 \oplus H_2 \oplus H_3 \oplus \ldots.$ We pick such a $g = \frac{3^{m+1} - 3}{2}$. Then, any non-zero element in $H_{m+1} \oplus H_{m+2} \oplus \ldots$ has absolute value $\frac{3^{m+2} - 1}{2}, \frac{3^{m+2}+1}{2},$ or $\geq 3^{m+2} - 1$. 

Furthermore, any term in $H_1 \oplus H_2 \oplus \ldots \oplus H_{m-1}$ has absolute value in $[4, \frac{3^{m+1}-5}{2}]$. With these bounds, we can see that we cannot have that the sum of a term in $H_1 \oplus \ldots \oplus H_{m-1}$, a term in $H_m$, and a term in $H_{m+1} \oplus H_{m+2} \oplus \ldots$ is $\frac{3^{m+1} - 3}{2}$. 

Hence, we have that $n [B_1] \cap  [B_1] = \emptyset$ for all $n \in G$, giving us that there does there exist a finite $F$ such that $n[B_1] \cap [B_1] \neq \emptyset$ for all $n \in \Gamma \backslash F$. Thus, $(X,\Gamma)$ is not strong mixing. 
\end{proof}

As the Chac\'{o}n System is weak mixing, but not strong mixing, it is a good example to work with to see if we can find a thick $N \subset \Gamma = \Z$ on which the system looks strong mixing. It turns out that for this particular system, such a thick $N$ exists. 

To show this, we need the following lemma, characterizing thick sets in $\Z$. 

\begin{lemma} 
\label{lem:thick}
A set $N \subset \Z$ is thick if and only if it contains arbitrarily long intervals (strings of consecutive integers). 
\end{lemma}
\begin{proof} If $N$ contains arbitrarily long intervals, then for any finite set $F \subset \Z$, if $n = \max\{|a| : a\in F\}$, then we can find an interval $I = \{c, c+1, \ldots, c+{2n}\} \subset N$ with $2n+1$ elements. It follows then that $c+n \in I \cap \bigcap_{\gamma \in F} \gamma I \subset N \cap \bigcap_{\gamma \in F} \gamma N$, and so $N$ is thick. 

Conversely, if $N$ does not contain an interval of length $n$, then for $F = \{1,2,\ldots, n\}$, $N \cap \bigcap_{\gamma \in F} \gamma N = \emptyset$, and so $N$ is not thick. 
\end{proof}

Now, we can return to the Chac\'{o}n System, with the following proposition. 

\begin{proposition} Let $(X, \Z)$ be the Chac\'{o}n System. Then, $(X, \Z)$ is thick strong mixing. \end{proposition}
\begin{proof} 
We first claim $N([B_k], [B_k])$ is thick for each $k$. This is evident from the description of the $H_i$'s and the characterization of thick sets (Lemma \ref{lem:thick}). To find an interval of length $m$ in $H_k \oplus H_{k+1} \oplus \ldots$, we notice that the set $$\bigoplus_{i = k}^{k+m-2} \left\{ \frac{3^{i+1}-1}{2}, \frac{3^{i+1} +1}{2}\right\}$$ is an interval of length $m$ (for $m \geq 2$) that is in $H_k \oplus H_{k+1} \oplus \ldots.$

We let $l_k$ be the length of the string $B_k$. Then, for any $-l_k \leq n,m \leq l_k$, we see that $$N(n[B_k], m[B_k]) \subset N([B_k], (-l_k) [B_k]) \cap N([B_k], (-l_k+1)[B_k]) \cap \ldots \cap N([B_k], l_k[B_k])=: \mathcal{M}_k,$$ the latter of which is still a thick set. For any $n, m, k$ in $\Z$, by choosing an $\ell \in \Z$ such that $l_\ell > |n-m|$, we see that $$N(n[B_k], m[B_k]) \supset N(n[B_\ell], m[B_\ell]) = N([B_\ell], (m-n)[B_\ell]) \supset \mathcal{M}_\ell.$$

As $[B_k] \subset [B_\ell]$ and $l_k \geq l_\ell$ for all $k \geq \ell$, we see that we have the containment $$\mathcal{M}_1 \supset \mathcal{M}_2 \supset \mathcal{M}_3 \supset \ldots$$ and each $\mathcal{M}_\ell$ is thick. This implies that for any $n[B_k]$ and $m[B_k]$, $N(n[B_k], m[B_k]) \supset \mathcal{M}_\ell$ for all $\ell$ large enough. 

Now, we can construct our set $N$ on which we have strong mixing as follows. Starting with $\mathcal{M}_1$, we pick an interval of length $i$ from each $\mathcal{M}_i$ and place those elements in the set $N$, which we can do as each $\mathcal{M}_i$ is thick. Then, we notice that for each $m \in \N$, $N$ contains only a finite number number of elements not in $\mathcal{M}_m$ (at most those chosen from $\mathcal{M}_1, \ldots, \mathcal{M}_{m-1}$) and $N$ is thick by construction. 

Now, for any non-empty open sets $U, V \subset X$, we had that $N(U, V) \supset N([B_\ell], m[B_\ell]) \supset \mathcal{M}_\ell$ for some $\ell$ and $m$ such that $|m| \leq l_\ell$, so that $nU \cap V \neq \emptyset$ for all $n \in N \backslash F$, where $F$ is a finite set of size at most $1 + 2 + \ldots + (\ell -1) = \frac{\ell(\ell-1)}{2}.$ Hence, $(X, \Z)$ is thick strong mixing along $N$. 
\end{proof}

As we have the result for the Chac\'{o}n System, we now work on generalizing these results to larger classes of spaces and groups. 

\subsection{Comments on the Chac\'{o}n System Example}

Here, we make a note of a few properties that made it possible for us to construct the thick set $N$ in the Chac\'{o}n System example. 

First, we had a growing collection of sets $C_k = \{(-l_k)[B_k], \ldots, [B_k], \ldots, l_k [B_k]\}$ such that each open set $U \subset X$ eventually contained some $U_k \in C_k$ for all $k$ large enough. This allowed us to construct a thick set $\mathcal{M}_k \subset \Gamma$ such that $N(U,V) \supset \mathcal{M}_k$ for all $U, V \in C_k$. 

As the $\mathcal{M}_k$ were all thick, we could strategically choose a few elements from each $\mathcal{M}_k$ to put into a set $N$ that later became the thick set on which we had thick strong mixing. This idea of building up out thick set $N$ from larger and larger pieces of the thick sets $\mathcal{M}_k$ will be a key component of our proofs in later sections.

\section{Equivalence of $k$-transitivity for all $k$ and Thick Strong Mixing}

We say that a system $(X, \Gamma)$ is \textbf{transitive} if for any two nonempty open sets $U, V \subset X$, there exists a $\gamma \in \Gamma$ such that $\gamma U \cap V \neq \emptyset$. Furthermore, we say that $(X, \Gamma)$ is \textbf{$k$-transitive} if $(X^k, \Gamma)$ is transitive. We notice that $k$-transitivity is equivalent to the property that $N(U_1, V_1) \cap \ldots \cap N(U_k, V_k)$ is nonempty for all collections of nonempty open sets $U_1, \ldots, U_k, V_1, \ldots, V_k$. From this definition, we see that weak mixing is equivalent to $2$-transitivity. 

In this section, we will prove the equivalence of $k$-transitivity for all $k$ and thick strong mixing (under some initial conditions), breaking the proof into three sections. We will then discuss some consequences of this equivalence. To restate, our goal is to prove the following: 

\textbf{Goal:} For systems $(X, \Gamma)$ where $X$ is second countable and $\Gamma$ is locally compact, second countable, but not compact, $k$-transitivity for all $k$ is equivalent to thick strong mixing.

\subsection{Forward Implication for Countable, Discrete Groups}
We start with the forward implication for countable discrete groups, which is a simpler case of one direction of our general theorem. To begin, we prove a lemma that is of a similar spirit to Theorem \ref{thm:tdt}. 

\begin{lemma}
\label{lem:trans}
A topological system $(X, \Gamma)$ is $k$-transitive for every $k \in \N$ if and only if for any set of nonempty open sets $\{U_1, \ldots, U_n, V_1, \ldots, V_n\}$ in $X$, we have that $\bigcap_{1 \leq i \leq n} N(U_i, V_i)$ is thick. 
\end{lemma}

\begin{proof}
We suppose that $(X, \Gamma)$ is $k$-transitive for each $k \in \N$. Then, if $\gamma_1,\ldots, \gamma_m$ are elements of $\Gamma$, we have that 
\begin{align*}
\bigcap_{\substack{1 \leq \ell \leq m \\ 1 \leq i \leq n}} \gamma_\ell N(U_i, V_i) = \bigcap_{\substack{1 \leq \ell \leq m \\ 1 \leq i \leq n}}  N(U_i, \gamma_\ell V_i) ,
\end{align*}
which is nonempty by the $\ell n$-transitivity of $(X, \Gamma)$ as each $\gamma_\ell V_i$ is open. For the reverse implication, given any collection of nonempty open sets $U_1, \ldots, U_k, V_1, \ldots, V_k$, $\bigcap_{1 \leq i \leq k} N(U_i, V_i)$ is thick and therefore is nonempty. So, $(X, \Gamma)$ is $k$-transitive. 
\end{proof}

The forward implication of this equivalence will be useful in our proofs of the equivalence of $k$-transitivity for all $k$ and thick strong mixing. Now, we state and prove the first of our propositions that will give us this equivalence. 

\begin{proposition}
\label{prop:cmc} For $X$ a second countable space and $\Gamma$ a countable, discrete group (that is not finite), if $(X, \Gamma)$ is $k$-transitive for all $k$, then it is thick strong mixing. 
\end{proposition}

\begin{proof}
As $X$ is a second countable space, we have a countable basis $\{U_i\}$ of open sets for the topology of $X$. Then, for each $i \in \N$, we define $C_i$ to be the finite set of open sets $$C_i = \{U_1, U_2, \ldots, U_i\}.$$ 
By definition, $C_1 \subset C_2 \subset C_3 \subset \ldots$. We let $$\mathcal{M}_n = \bigcap_{1 \leq i, j \leq n} N(U_i, U_j),$$ which is thick by the previous lemma, Lemma \ref{lem:trans}, as the system is $k$-transitive for all $k$. Furthermore, as $C_1 \subset C_2 \subset C_3 \subset \ldots$, we see that we have the inclusion $$\mathcal{M}_1 \supset \mathcal{M}_2 \supset \mathcal{M}_3 \supset \ldots.$$ Now, as $\Gamma$ is countable, we can enumerate the elements $\Gamma = \{\gamma_1, \gamma_2, \gamma_3, \ldots\}$. Since each $\mathcal{M}_m$ is thick, for each $m$ we have that $\mathcal{M}_m \cap \bigcap_{i=1}^m \gamma_i \mathcal{M}_m$ is non-empty. Thus, we can find an element $\delta_m$ in this intersection and define the finite set $$N_m = \{\delta_m, \gamma_1^{-1} \delta_m, \ldots, \gamma_m^{-1} \delta_m\} \subset \mathcal{M}_m.$$ We then let $N = \bigcup_{m=1}^\infty N_m$ and claim then that $N$ is thick and that $(X, N)$ is strong mixing. 

We begin by proving thickness. As $\Gamma$ is discrete, compact sets are finite sets. For any compact set $\{\gamma_{i_1}, \ldots, \gamma_{i_n}\} \subset \Gamma$, we have that $\{\gamma_{i_1}, \ldots, \gamma_{i_n}\} \subset \{\gamma_1, \ldots, \gamma_m\}$ for some $m$. Therefore, we have that 
$$\delta_m \in \left( N_m \cap \bigcap_{i=1}^m \gamma_i N_m \right) \subset \left( N \cap \bigcap_{i=1}^m \gamma_i N \right)\subset \left(N \cap \bigcap_{j=1}^n \gamma_{i_j} N\right),$$
showing that $N$ is thick. We also notice that for any two open sets $\tilde U$ and $\tilde V$ in $X$, there are $U \subset \tilde U$ and $V \subset \tilde V$ with $U, V \in C_m$ for all $m \geq n$ for some $n$. Then, for any $\gamma \in N \backslash \bigcup_{i=1}^{n-1} N_i$, we have that $\gamma \in N_k \subset \mathcal{M}_k$ for some $k \geq n$. So, $\gamma U \cap V \neq \emptyset$ and therefore $\gamma \tilde U \cap \tilde V$ is nonempty. Thus, $(X, \Gamma)$ is thick strong mixing. 
\end{proof}

Having established the equivalence of $k$-transitivity for all $k$ and thick strong mixing for countable, discrete groups, we now move on to the arbitrary locally compact, second countable groups. 

\subsection{Forward Implication for Locally Compact, Second Countable Groups}

We now move away from the realm of discrete countable groups and suppose that $\Gamma$ is just a second countable, locally compact topological group. Then, as $\Gamma$ is locally compact and Hausdorff, it is also a regular space and Urysohn's Lemma then gives us that $\Gamma$ is metrizable. Thus, we may assume that $\Gamma$ comes equipped with a metric. Furthermore, that $\Gamma$ is second countable implies that it is Lindeloff, and metrizable and Lindeloff together imply that $\Gamma$ is also separable. 

As $\Gamma$ is now not a discrete group, we need to modify our method slightly from the previous section. The idea is that we will again define the $C_i$ and find thick sets $\mathcal{M}_1 \supset \mathcal{M}_2 \supset \mathcal{M}_3 \supset \ldots$ for which $N(U,V) \supset \mathcal{M}_i$ for all $U, V \in C_i$. Then, we can find a countable dense set $\{\gamma_1, \gamma_2, \ldots\} \subset \Gamma$ and take advantage of the continuous property of the $\Gamma$-action, to construct our desired thick set $N$ by looking at neighborhoods of the $\gamma_i$. This way, we can again restrict our attention to a countable set. 

\begin{proposition}
\label{prop:asclc}
Let $X$ be a second countable topological space and $\Gamma$ a continuous action on $X$, where $\Gamma$ is second countable, and locally compact. Then, if $(X, \Gamma)$ is $k$-transitive for all $k$, it is also thick strong mixing. 
\end{proposition}

\begin{proof}
As in the proof of Proposition \ref{prop:cmc}, we can let $C_i = \{U_1, \ldots, U_i\}$, where $\{U_i\}_{i=1}^\infty$ is a countable basis for the topology of $X$. Then, we have that $$\mathcal{M}_n = \bigcap_{1 \leq i, j \leq n} N(U_i, U_j)$$ is thick for all $n \in \N$. 

By the metrizability and separability of $\Gamma$, we can also find a countable basis of balls $\{\mathcal{U}_i = B(g_i, r_i)\}_{i \in \N}$ for $\Gamma$. By the local compactness of the space, the $\mathcal{U}_i$ may be chosen so that each $\overline{\mathcal{U}_i}$ is compact. Now the idea is that we want to find compact sets $N_i \subset \Gamma$ such that $$N_i \cap \gamma_1 N_i \cap \gamma_2 N_i \cap \ldots \cap \gamma_i N_i \neq \emptyset, \text{ for all }\gamma_j \in \mathcal{U}_j$$ for which $\mathcal{M}_i \supset N_i$. If we then took $N = N_1 \cup N_2 \cup \ldots$, we would have that $N$ is thick as for any $\gamma_1, \ldots, \gamma_i \in \Gamma$, we can find $j_1 < j_2 < \ldots < j_i$ for which $\gamma_1 \in \mathcal{U}_{j_1}, \ldots, \gamma_i \in \mathcal{U}_{j_i}$. Then, $$\left(N \cap \gamma_1 N \cap \ldots \cap \gamma_i N\right) \supset \left(N_{j_i} \cap \gamma_1 N_{j_i} \cap \ldots \cap \gamma_i N_{j_i}\right) \neq \emptyset,$$
and so we would have the thickness of $N$. Furthermore, for any two open sets $U, V$, we could find an $I$ such that for all $i \geq I$, there exist $U_I, V_I \in C_I \subset C_i$ such that $U \supset U_I, V \supset V_I$. Then, for all $i \geq I$, we have that $N_i \subset N(U,V)$. Hence, $\gamma U \cap  V \neq \emptyset$ for all $\gamma \in N \backslash F$ where $F = N_1 \cup \ldots \cup N_{I-1}$. But each $N_i$ was chosen to be compact so $F$ is compact, and we have the mixing condition on $N$. 

Now, we need to construct the sets $N_i$. To do this, we will take advantage of the structure and topology of $\Gamma$ as well as that $\Gamma$ acts continuously on $X$. 

By the thickness of each $\mathcal{M}_i$, for every $(\gamma_1, \gamma_2, \ldots, \gamma_i) \in \overline{\mathcal{U}_1} \times \overline{\mathcal{U}_2} \times \ldots \times \overline{\mathcal{U}_i}$, we can find $m_0, m_1, \ldots, m_i \in \mathcal{M}_i$ for which $$m_0 =  \gamma_1 m_1 = \ldots = \gamma_i m_i.$$


Around each $m_j, 0 \leq i \leq i$, we can find a ball $B_{m_j} := B(m_j, r_j) \subset \mathcal{M}_i,$ $r_j > 0$, with compact closure as the $\mathcal{M}_i$ are open by the continuity of the map $X \times \Gamma \rightarrow X$. Then, as $\Gamma$ is a topological group, the map $\varphi: \Gamma \times \Gamma \rightarrow \Gamma$ defined by $\varphi(\gamma_1, \gamma_2) = \gamma_1 \gamma_2$ is continuous, so around each $\gamma_j$, we can find a nonempty open ball $B_{\gamma_j}$ such that $\varphi(B_{\gamma_j} \times B_{m_j}) \subset B_{m_0}$ for each $j$, $1 \leq j \leq i$. 

Now, if we let 
\begin{align*} 
B_{(\gamma_1, \ldots, \gamma_i)} & = \overline{B_{m_0}} \cup \ldots \cup \overline{B_{m_i}} \\
N_{(\gamma_1, \ldots, \gamma_i)} & = B_{\gamma_1} \times \ldots \times B_{\gamma_i},
\end{align*} 
we have that for any $(\tilde \gamma_1, \ldots, \tilde \gamma_i) \in N_{(\gamma_1, \ldots, \gamma_i)}$, there exists $\tilde m_0, \ldots, \tilde m_i \in B_{(\gamma_1, \ldots, \gamma_i)}$ such that $$\tilde m_0 = \tilde \gamma_1 \tilde m_1 = \ldots = \tilde \gamma_i \tilde m_i.$$
Then, the sets $N_{\gamma_1, \ldots, \gamma_i}$ ranging over all $(\gamma_1, \ldots, \gamma_i)$ in $\overline{U_1} \times \ldots \times \overline{U_i}$ form an open cover. Since $\overline{\mathcal{U}_1} \times \ldots \overline {\mathcal{U}_i}$ is compact, we can then find a finite subcover of this space. Letting $N_i$ be the finite union of $B_{(\gamma_1, \ldots, \gamma_i)}$ for the $(\gamma_1, \ldots, \gamma_i)$ that generated the open sets in this cover, we have that $N_i$ is compact, thick, and in $\mathcal{M}_i$, which is what we wanted. 
\end{proof}


\subsection{Tying it Together: the Reverse Implication}

As of this point, we have established that given sufficient conditions on a system $(X, \Gamma)$, $k$-transitivity for all $k$ then implies that $(X, \Gamma)$ is thick strong mixing. We wonder then if this implication is actually an equivalence. That is, given a thick strong mixing system, is it then $k$-transitive for all $k$? We answer this question in the affirmative in this section (again after assuming sufficient initial conditions). 

Before we state our proposition, we note that if $\Gamma$ is a compact group, then $(X, \Gamma)$ is trivially thick strong mixing as for any open $U, V \subset X$, we have that for all $\gamma \in \Gamma \backslash \Gamma = \emptyset$, $\gamma U \cap V \neq \emptyset$. Therefore, we want to exclude compact groups $\Gamma$ from consideration. Doing so, we will have a reverse implication. We first start with a lemma. 

\begin{lemma} 
\label{lem:notthick}
If $C$ is a compact, proper subset of $\Gamma$ a topological group, then $C$ does not contain a thick subset of $\Gamma$. 
\end{lemma}
\begin{proof} As $C$ is compact and $\Gamma$ is Hausdorff, $C$ is also closed in $\Gamma$. Since $C$ is a proper subset of $\Gamma$, letting $U = \Gamma / C$, we have that $U$ is open and that there exists some element $u \in U$. 

For every $c \in C$, there exists a group element $\gamma_c = uc^{-1}$ such that $\gamma_c c = u$. By continuity, $\gamma_c^{-1} U$ is open, and contains $c$. Then, $\{\gamma_c^{-1} U\}_{c \in C}$ is an open cover of $C$ and therefore has a finite subcover $\{U_i = \gamma_i^{-1} U\}_{i=1}^n$. 

But now, as the action of each $\gamma : \Gamma \rightarrow \Gamma$ is bijective, we know that $\gamma_i^{-1} C \cap U_i = \emptyset$, as $U_i$ is contained in the image of $U$ under $\gamma_i^{-1}$, and $U \cap C = \emptyset$. Thus, we have that 
$$C \cap \bigcap_{i=1}^n \gamma_i^{-1} C = C \cap  \left( \bigcup_{i=1}^n U_i \right) \cap \left( \bigcap_{i=1}^n \gamma_i^{-1} C \right) = \emptyset.$$ Hence, $C$ cannot contain a thick subset of $\Gamma$. 
\end{proof}

Now, we have the tools that we need to prove the reverse implication. 

\begin{proposition} 
\label{prop:forward} 
If $(X, \Gamma)$ is thick strong mixing where $X$ is a second countable space and $\Gamma$ is second countable and locally compact but not compact, then $(X, \Gamma)$ is a $k$-transitive for all $k \in \N$. 
\end{proposition}
\begin{proof} We suppose that $(X, \Gamma)$ satisfies the conditions in the statement of the proposition. Then, there exists a thick $N \subset \Gamma$ on which we have strong mixing. From Lemma \ref{lem:notthick}, we know that $N$ cannot be compact. We take any $k$ pairs of nonempty open sets $U_i$ and $V_i$ in $X$. By the thick strong mixing property, for each $i$ there exists a compact $C_i \subset N$ such that $\gamma U_i \cap V_i \neq \emptyset$ on $N \backslash C_i$. This implies that $\gamma U_i \cap V_i \neq \emptyset$ for all $i$ on $N \backslash (C_1 \cup \ldots \cup C_k)$. But $C_1 \cup \ldots \cup C_k$ is compact and $N$ is not compact, so $N \backslash (C_1 \cup \ldots \cup C_k) \neq \emptyset$ and there exists a $\gamma \in N \subset \Gamma$ such that $\gamma U_i \cap V_i \neq \emptyset$ for all $i, 1 \leq i \leq k$, implying $k$-transitivity. 
\end{proof}

Now, we have both the forward and reverse implication in the equivalence of $k$-transitivity for all $k$ and thick strong mixing. If we include in the result coming from Lemma \ref{lem:trans}, we get the main theorem, Theorem \ref{thm:main}.

\subsection{Consequences of the Main Theorem}

We conclude this section with some consequences of the main theorem, Theorem \ref{thm:main}. We recall that our motivating problem was to find when a system being weak mixing then implies that it is thick strong mixing. Now, we can resolve this problem for abelian groups. It is well known that for actions of abelian groups, weak mixing implies $k$-transitivity for all $k$. Then, by Theorem \ref{thm:main}, the equivalence of $k$-transitivity for all $k$ and thick strong mixing, we have the following result. 

\begin{corollary}
\label{cor:abelian}
If $(X, \Gamma)$ is a system such that $X$ is second countable and $\Gamma$ is locally compact but not compact, second countable, and abelian, then weak mixing is equivalent to thick strong mixing. 
\end{corollary}

From this result, we can immediately deduce that the Chac\'{o}n system of Section \ref{sec:chacon} is thick strong mixing. We also see that this corollary gives that for every $\Z$-action on a second countable space $X$, weak mixing is equivalent thick strong mixing. \\

\begin{remark} We recall that the problem of when is a weak mixing system thick strong mixing was motivated by the result in ergodic theory that given sufficient initial conditions, a weakly mixing transformation is strong mixing on a dense subset of $\N$. Our result that $\Z$-actions on a second countable space satisfy that weak mixing implies thick strong mixing can then be thought of as an analogue of this result, adding to the list of parallels between topological dynamics and ergodic theory.
\end{remark}



\section{Nonabelian Examples}

In this section, we will look at some examples of nonabelian group actions that are weakly mixing. By looking at whether each system is $k$-transitive for all $k$, we will be able to determine whether each system satisfies the thick strong mixing property. In particular, we will construct both nonabelian group actions that are and are not thick strong mixing, showing that weak mixing does not always imply thick strong mixing in the case of nonabelian groups.

\subsection{$SL(2, \Z)$ Acting on the Torus, $\R^2 / \Z^2$}

We consider the following example. We let $\T^2 = \R^2 / \Z^2$, which is isomorphic to the torus, and let our acting group be $SL(2 , \Z)$. As $SL(2, \Z)$ is generated by the matrices $$\begin{bmatrix} 1 & 1 \\ 0 & 1 \end{bmatrix} \text{ and } \begin{bmatrix} 0 & 1 \\ -1 & 0 \end{bmatrix},$$ and each of these matrices can be seen to induce an automorphism on the torus (where the action on $\R^2$ is by matrix multiplication), this is a well-defined action. Furthermore, we have the following. 

\begin{proposition} $(\T^2, SL(2, \Z))$ is weakly mixing but not mixing.
\end{proposition}
\begin{proof}

We first show that $(\T^2, SL(2, \Z))$ is weak mixing. We suppose that we have open sets $U_1, U_2, U_3, U_4 \subset \T^2$. Then, if we represent the torus by its fundamental domain $[0,1) \times [0,1)$ in $\R^2$, we have that there exist nonempty $V_i = (a_i, b_i) \times (c_i, d_i) \subset U_i$ for each $i$. 

Then, for $c \in \Z$ such that $c > 1/(b_1 - a_1)$ and $c > 1/(b_3 -a_3)$, we have that $(c \cdot a_i, c \cdot b_i)$ covers $[0,1) \pmod{1}$ for $i=1,3$. Thus, there exist blocks $$V_1^\prime = (a_1^\prime, b_1^\prime) \times (c_1^\prime, d_1^\prime) \subset \left(\begin{bmatrix} 1 & 0 \\ c & 1 \end{bmatrix} V_1\right) \cap \left((a_1, b_1) \times (c_2, d_2)\right),$$
and similarly for $V_3^\prime$. This corresponds to a vertical shearing of the set. Then, choosing $d > 1/\min\{(d_1^\prime - c_1^\prime), (d_3^\prime - c_3^\prime)\}$, we have that for $$M = \begin{bmatrix} 1 & d \\ 0 & 1 \end{bmatrix} \begin{bmatrix} 1 & 0 \\ c & 1 \end{bmatrix},$$ that $M V_1 \cap  V_2 \neq \emptyset$ and $M V_3 \cap  V_4 \neq \emptyset$. Thus, $(\T^2, SL(2, \Z))$ is weak mixing. 

However, $(\T^2, SL(2, \Z))$ is not mixing. Again representing the torus by the fundamental domain $[0,1) \times [0,1)$, we consider the sets $V = (1/5, 2/5) \times (1/5, 2/5)$ and $U = (3/5, 4/5) \times (3/5, 4/5)$. Then, if we let $$g_a = \begin{bmatrix} 1 & a \\ 0 & 1 \end{bmatrix},$$ where $a \in \Z$, then each $g_a \in SL(2, \Z)$. Furthermore, $g_a U \subset [0,1) \times (1/5, 2/5)$ for all $a$. Thus, $g_a U \cap V = \emptyset$ for all $a \in \Z$. As this is an infinite set, it follows then that there is no compact (which implies finite) set $F \subset SL(2, \Z)$ for which $\gamma U \cap V \neq \emptyset$ for all $\gamma \in SL(2, \Z) \backslash F$, so $(\T^2, SL(2, \Z))$ is not mixing. 
\end{proof}

We turn our attention now to the question that we are interested in, whether there is a thick subset of $SL(2, \Z)$ on which the system is thick strong mixing. It suffices now to check whether $(\T^2, SL(2, \Z))$ is $k$-transitive for every $k$. This is an adaptation of the weak mixing proof from the previous proposition. 

\begin{proposition} Given $(\T^2, SL(2, \Z))$, the $k$-fold product system is topologically transitive for each $k \in \N$ and so $(\T^2, SL(2, \Z))$ is thick strong mixing. 
\end{proposition}

\begin{proof}
We suppose that we have $U_1, \ldots, U_{2k}$, open subsets of $\T^2$. Then, representing $\T^2$ by its fundamental domain $[0,1) \times [0,1)$, we can find nonempty sets $$V_i = (a_i, b_i) \times (c_i, d_i) \subset U_i$$ for each $i, 1\leq i \leq 2n$. Taking $c  > 1 / \min\{(b_{1}  - a_1, b_3 - a_3, \ldots, b_{2n-1} - a_{2n-1}\}, c \in \Z$, we have that $(c \cdot a_i, c \cdots b_i)$ covers $[0,1) \pmod{1}$ for each $i = 1, 3, \ldots, 2k-1$, so there exist blocks  $$V_{2i-1}^\prime = (a_{2i-1}^\prime, b_{2i-1}^\prime) \times (c_{2i-1}^\prime, d_{2i-1}^\prime) \subset \left(\begin{bmatrix} 1 & 0 \\ c & 1 \end{bmatrix} V_{2i}\right) \cap \left((a_{2i-1}, b_{2i-1}) \times (c_{2i}, d_{2i})\right),$$
for each $i, 1 \leq i \leq k$. Then, choosing $d > 1/\min_{1 \leq i \leq k} \{d_{2i-1}^\prime - c_{2i-1}^\prime\}$, we have that $(d \cdot c^\prime_{2i-1}, d \cdot d^{\prime}_{2i-1})$ covers $[0,1) \pmod{1}$ for each $i, 1 \leq i \leq k$. Thus, we have that for $$M = \begin{bmatrix} 1 & d \\ 0 & 1 \end{bmatrix} \begin{bmatrix} 1 & 0 \\ c & 1 \end{bmatrix},$$ that $M V_{2i-1} \cap V_{2i} \neq \emptyset$ for each $i, 1 \leq i \leq k$, showing that $$N(U_1 \times \ldots \times U_{k}, U_{k+1} \times \ldots \times U_{2k}) \neq \emptyset,$$ giving us topological transitivity of $(X^k, \Gamma)$. 

Now, by Theorem \ref{thm:main}, we have that $(\T^2, SL(2, \Z))$ is thick strong mixing. 
\end{proof}

This gives us an example of a system that is weak mixing and thick strong mixing, but not strong mixing. We also want to check if there exist systems that are weak mixing but not thick strong mixing. In the setting of abelian group actions, this could not happen (given the standard conditions on $X$ and $\Gamma$). In the next section, we will construct such an action.

\subsection{M\"{o}bius Transformations on the Riemann Sphere}

In this section, we answer the question of whether every topologically weak mixing system $(X, \Gamma)$ where $X$ is second countable and $\Gamma$ is locally compact and second countable has a thick subset $N$ of $\Gamma$ on which the system is mixing. As we will see, the answer to this question will be no. As an example of when this is not the case, we will consider the action of $GL(2, \C)$ on the Riemann sphere $\C P^1 \simeq \C \cup \{\infty\} \simeq S^2$ given by the m\"obius transformations. 

That is, given a matrix $$\gamma = \begin{bmatrix} a & b\\ c & d \end{bmatrix} \in GL(2, \C),$$ and $z \in \C$, we define $$\gamma (z) = \frac{az + b}{cz+d},$$ and $\gamma (\infty) = a/c$ if $c \neq 0$ and $\infty$ otherwise. It can be checked that this indeed defines a continuous action of $GL(2, \C)$ on $\C P^1$. Furthermore, the following theorem is well-known: 

\begin{theorem} Given any three points $z_1, z_2, z_3 \in \C P^1$, there is a unique mobius transformation that takes $z_1$ to $0$, $z_2$ to $1$, and $z_3$ to $\infty$, given by the following matrix (and non-zero scalar multiples): 
$$M_{(z_1, z_2, z_3)} := \begin{bmatrix} z_2 - z_3 & - z_1(z_2 - z_3) \\ z_2 - z_1 & -z_3(z_2-z_1)\end{bmatrix}.$$
\end{theorem} 

We notice that this immediately implies that $(\C P^1, GL(2, \C))$ is $3$-transitive, and therefore $2$-transitive, and therefore weak mixing. If we had that $(\C P^1, GL(2, \C))$ were $k$-transitive for every $k$, we would then have that the system is also strong mixing on a thick subset of $GL(2, \C)$. However, as we will now see, this is not the case. 

\begin{proposition}
\label{prop:4trans}
 $(\C P^1, GL(2, \C))$ is $3$-transitive but not $4$-transitive. 
\end{proposition}
\begin{proof} As m\"{o}bius transformations are invertible, we have that given any two triples of points $(z_1, z_2, z_3)$ and $(w_1, w_2, w_3)$ in $(\C P^1)^3$, there is a unique m\"{o}bius transformation that takes each $z_i$ to $w_i$, represented by the matrix $M^{-1}_{(w_1, w_2, w_3)}M_{(z_1, z_2, z_3)}.$ This then implies $3$-transitivity. 

To show that we cannot have $4$-transitivity, we consider the following counterexample. We let $U_i = V_i = B(i, \epsilon), 1 \leq i \leq 3$, for an $\epsilon > 0$ to be defined later. Then, any m\"{o}bius transformation $\gamma$ such that $\gamma U_i \cap V_i \neq \emptyset$ for all $i$ corresponds to a scalar multiple of the matrix $$M = M^{-1}_{(w_1, w_2, w_3)}M_{(z_1, z_2, z_3)}$$ for $z_i \in U_i, w_i \in V_i$. Explicitly writing out this matrix, we get that 
\begin{align*}
M = \begin{bmatrix} - w_3(w_2 - w_1) & w_1(w_2 - w_3) \\ -(w_2 - w_1) & w_2 - w_3 \end{bmatrix} \begin{bmatrix} z_2 - z_3 & - z_1 (z_2 - z_3) \\ z_2 - z_1 & - z_3(z_2 - z_1) \end{bmatrix} = \begin{bmatrix} a & b \\ c & d \end{bmatrix},
\end{align*}
where 
$$\begin{cases} 
a & = w_3(w_2-w_1)(z_3-z_2) - w_1(w_3-w_2)(z_2-z_1) \\
b & =  -w_3z_1(w_2 - w_1)(z_3-z_2) + w_1z_3(w_3-w_2)(z_2-z_1) \\
c & = (w_2-w_1)(z_3-z_2) - (w_3-w_2)(z_2-z_1) \\
d & = -z_1(w_2-w_1)(z_3-z_2) + z_3(w_3-w_2)(z_2-z_1).
\end{cases}$$
Understanding $z = o(\epsilon)$ to mean that $|z| \leq \epsilon$, we have that $w_i = i + o(\epsilon)$ for $i = 1,2,3$ and $w_i - w_{i-1} = 1 + o(2\epsilon)$ for $i=2,3$, and similarly for $z_i$. As a result, there is a constant $c > 0$ such that for all $\epsilon$ small enough, $$M = \begin{bmatrix} 2 + o(c\epsilon) & o(c \epsilon) \\ o(c \epsilon) & 2 + o(c \epsilon) \end{bmatrix}.$$
Then, letting $\gamma$ be the m\"{o}bius transformation associated with $M$, we have that given $\epsilon^\prime > 0$, for any $|z| \leq \epsilon^\prime$, we have that 
$$|\gamma(z)| = \left| \frac{az+b}{cz+d} \right| \leq \frac{(2 + c\epsilon )\epsilon^\prime + c \epsilon}{- c \epsilon \epsilon^\prime + 2 - c \epsilon} \leq \frac{3}{1} = 3,$$ for small enough $\epsilon$ and $\epsilon^\prime$. Thus, there does not exist a $\gamma \in GL(2, \C)$ such that $\gamma U_i \cap V_i \neq \emptyset$ for $U_i$ and $V_i$, $1 \leq i \leq 3$, as defined above and $U_4 = B(0, \epsilon^\prime), V_4 = B(4, 1/2)$. Hence, $(\C P^1, GL(2, \C))$ is not $4$-transitive. 
\end{proof}

Since $k$-transitivity for all $k$ was equivalent to thick strong mixing, the previous proposition shows that this system $(\C P^1, GL(2, \C))$ is weak mixing but not thick strong mixing.

With this result, we have shown that there exist topologically weak mixing systems $(X, \Gamma)$ where $X$ is second countable and $\Gamma$ is second countable, locally compact for which there does not exist a thick subset of $\Gamma$ on which we have strong mixing, when $\Gamma$ is nonabelian. 

We can ask then if we can find certain classes of nonabelian groups for which we do have that a weakly mixing group action is then thick strong mixing. This is one of the questions that we will investigate in the next section.                                                          

\section{Extensions and Further Directions}

In this section, we will try to extend our understanding of thick strong mixing by pursuing two different lines of inquiry: 
\begin{enumerate}
\itemsep0em 
\item We found that in topological dynamics an analogue of Theorem \ref{thm:mdo2} (weak mixing implies strong mixing on a density one subset) held for abelian groups. However, there was a precursor to this theorem, Theorem \ref{thm:mdo1}, that said that for any two measurable sets $A$ and $B$, $$\lim_{n \in J, n \rightarrow \infty} \mu (T^{-n} A \cap B) = \mu (A) \mu (B)$$ for some density one subset $J$ that may depend on $A$ and $B$. Can we find an analogue of this theorem in topological dynamics? That is, when does there exist a thick set $N_{A, B}$ for each pair of nonempty open $A, B$ such that $\gamma A \cap B \neq \emptyset$ for all $\gamma \in N_{A,B} \backslash F$ for some compact $F$?

\item Since we have that under suitable initial conditions the weak mixing action of an abelian group is automatically thick strong mixing, a natural question to ask is the following: are there larger classes of groups for which this implication (weak mixing implies thick strong mixing) holds? In particular, it makes sense to look at classes of groups that are ``almost" abelian in some way.
\end{enumerate}

We will see that the investigation of both of these problems will naturally lead us to consider the implications between various notions of transitivity, including which notions of transitivity pass down to finite index subgroups. 

\subsection{Intersections on Thick Sets} 

We begin by looking at the first problem stated above. We found in an earlier section a characterization of when a weakly mixing action was also thick strong mixing, allowing us to conclude that a result like Theorem \ref{thm:mdo2} existed in topological dynamics. 

To find an analogue of Theorem \ref{thm:mdo1}, we ask for which systems $(X, \Gamma)$ do with have a thick set for each pair of nonempty open sets $A$, $B$ in $X$ (possibly dependent on $A$ and $B$) on which we have the strong mixing property. That is, we ask for which systems $(X, \Gamma)$ does there exist for every pair of nonempty open sets $A, B \subset X$ a thick set $N_{A,B}$ such that $\gamma A \cap B \neq \emptyset$ for all $\gamma \in N_{A,B} \backslash F$ for a compact set $F$. As thick sets cannot be contained in a compact set by Lemma \ref{lem:notthick}, we can see that $N_{A,B} \backslash F \subset N(A,B)$ and is thick so this is equivalent to the following question. 

\textbf{Question:} For which systems $(X, \Gamma)$ is $N(A,B)$ thick for every pair of nonempty open sets $A$ and $B$ in $X$. 

We will show that the property of $N(A, B)$ being thick for all nonempty open $A$ and $B$ in $X$ is equivalent to the action of $\Gamma$ on $X$ being elastic, a notion that we will define here. The action of $\Gamma$ on $X$ is said to be \textbf{elastic} if for any finite collection of nonempty open sets $U, V_1, V_2, \ldots, V_n$, there exists a $\gamma \in \Gamma$ such that $\gamma U \cap V_i \neq \emptyset$ for all $i, 1 \leq i \leq n$. In our next proposition, we state and prove the aforementioned equivalence. 

\begin{theorem} For a system $(X, \Gamma)$, the property that $N(A, B)$ is thick for all nonempty open $A, B \subset X$ is equivalent to the elasticity of the action of $\Gamma$ on $X$. 
\end{theorem}
\begin{proof}
We suppose first that we have the elasticity of the action of $\Gamma$ on $X$. Then, for any pair of nonempty open sets $A, B \subset X$ and finite set of $\gamma_1, \ldots, \gamma_n$, we have that
\begin{align*}
N(A, B) \cap \gamma_1 N(A, B) \cap \ldots \cap \gamma_n N(A,B) = N(A, B) \cap N(A, \gamma_1 B) \cap \ldots \cap N(A, \gamma_n B) \neq \emptyset
\end{align*}
by the elasticity of the action of $\Gamma$. 

Conversely, we suppose that $N(A, B)$ is thick for all nonempty open $A$ and $B$ and that we had some nonempty open sets $U, V_1, \ldots, V_n$. Now, we can find a $\gamma_2$ such that $\gamma_2^{-1} V_2 \cap V_1 =: \tilde V_2 \neq \emptyset$. Repeating this process inductively for $i = 3, \ldots, n$, we can find $\gamma_i$ such that $\gamma_i^{-1} V_i \cap \tilde V_{i-1} =: \tilde V_i \neq \emptyset$. Then, setting $V = \tilde V_n$, we have that $\gamma_i V \subset V_i$ for all $i, 2 \leq i \leq n$. 

Then, we have that $N(U, \gamma_i V) \subset N(U, V_i)$ for $2 \leq i \leq n$. It follows then that 
\begin{align*} 
N(U, V_1) \cap \ldots \cap N(U, V_n) & \supset N(U, V) \cap N(U, \gamma_2 V) \cap \ldots \cap N(U, \gamma_n V) \\
& = N(U, V) \cap \gamma_2 N(U, V) \cap \ldots \cap \gamma_n N(U, V) \\
& \neq \emptyset,
\end{align*} 
as $N(U, V)$ is thick. Thus, the action of $\Gamma$ is elastic. 
\end{proof}

Now that we have the equivalence of the thickness of $N(A, B)$ for all nonempty open $A$ and $B$ and the property of elasticity, we can use known results about elasticity to help us better understand for which systems $N(A, B)$ is always thick. Cairns et. al. give the following implication graph in their paper: 

\begin{theorem}[\cite{C}, Theorem 1] Let $(X, \Gamma)$ be a topological system where $\Gamma$ is an infinite group and $X$ a Hausdorff topological space. Then the following implications hold. 

\begin{center}
\begin{tikzpicture}
\node at (5,5) {\text{strongly mixing}};
\draw[->] (5,4.5) -- (5,4);
\node at (5, 3.5) {\text{$k$-transitive for all $k$}};
\draw[->] (3.5, 3) -- (2.5, 2.5);
\draw[->] (6.5, 3) -- (7.5, 2.5);
\node at (2, 2) {\text{weakly mixing}};
\node at (8, 2) {\text{elastic}};
\draw[->] (2.5, 1.5) -- (3.5, 1);
\draw[->] (7.5, 1.5) -- (6.5, 1);
\node at (5, 0.5) {\text{transitive}};
\end{tikzpicture}
\end{center}
\end{theorem} 

As a result, we see that $N(A, B)$ being thick (which we recall is equivalent to elasticity) is independent from weak mixing. However, as for abelian groups $\Gamma$ weak mixing implies $k$-transitivity for all $k$, we have the following corollary. 

\begin{corollary} For systems $(X, \Gamma)$ where $\Gamma$ is an abelian group, weak mixing implies that $N(A, B)$ is thick for all nonempty open sets $A, B \subset X$. 
\end{corollary} 

This extends the implication $(i) \Rightarrow (iv)$ in Theorem \ref{thm:tdt} to a larger class of groups and spaces (beyond compact spaces and discrete, countable, abelian groups) and gives us a counterpart to Theorem \ref{thm:mdo1} in topological dynamics, after replacing the notion of density one with that of thickness. 

\subsection{Some ``Almost" Abelian Groups}

We also want to know for which nonabelian groups $\Gamma$ does a system $(X, \Gamma)$ being weak mixing imply that it is also thick strong mixing. As we know that for abelian groups this implication always holds, it is natural to first look at classes of groups that are ``almost" abelian. We will start out by looking at virtually abelian groups. 

We say that a group $\Gamma$ is \textbf{virtually abelian} if it contains a finite order subgroup $N \subset \Gamma$ such that $N$ is abelian. For example, any semidirect product $N \rtimes H$ for $N$ an abelian group and $H$ a finite group is virtually abelian, as the subgroup consisting of elements $(n,e)$ for $n \in N$ and $e$ the identity in $H$ is then an abelian group of index $|H|$. 

Another way that a group can be ``almost" abelian if it can be built up from an abelian group by a finite number of abelian extensions. This motivates us to look at the class of nilpotent group. We say that a group $\Gamma$ is \textbf{nilpotent} if its lower central series is finite: $$\Gamma = \Gamma_1 \trianglerighteq \Gamma_2 \trianglerighteq \ldots \trianglerighteq \Gamma_n = \{1\},$$ where each $\Gamma_{i+1} = [\Gamma_i, \Gamma]$, the subgroup of $\Gamma$ generated by all commutators $[x,y]$ where $x$ is in $\Gamma_i$ and $y$ is in $\Gamma$. We notice here that $\Gamma_{n-1}$ is abelian and each $\Gamma_i/\Gamma_{i+1}$ is abelian. 

To attack the problem of whether the weak mixing actions of virtually abelian groups are necessarily thick strong mixing, we might hope to follow a line of reasoning as follows: if the weak mixing of the action of $\Gamma$ implied that the action of the finite index abelian subgroup $N$ is also weak mixing, the action of $N$ would be $k$-transitive for all $k$. This would then pass up to the whole group $\Gamma$, making $(X, \Gamma)$ $k$-transitive for all $k$ and we could apply our main theorem, Theorem \ref{thm:main}, to say that $(X, \Gamma)$ is thick strong mixing. 

Similarly, we might hope in the nilpotent case that the weak mixing property of the action of $\Gamma_1 = \Gamma$ might pass down to $\Gamma_2, \Gamma_3, \ldots, \Gamma_{n-1}$ and we could use the abelian property of $\Gamma_{n-1}$ to make a similar argument about passing $k$-transitivity up to the whole group $\Gamma$. 

However, we will see that this approach seems infeasible as there are examples of weak mixing group actions $(X, \Gamma)$ and a normal subgroup $N \trianglelefteq \Gamma$ of index $2$ such that $(X, N)$ is no longer weak mixing. 

\subsection{Some Results about Mixing and Subgroups}

Here, we will give a brief overview of some of the things that are known about various mixing properties passing to subgroups. We will then apply these results to a discussion of our problem. 

The first result of interest to us is that weak mixing does not pass down to finite index subgroups. 

\begin{example}[\cite{C}, Proposition 5]\label{ex:homeo} If $(X, \Gamma)$ is weak mixing and $N \subset \Gamma$ is a subgroup of finite index, then $(X, N)$ need not be weak mixing. Consider $X = \R$, and $\Gamma$ to be the group of homeomorphisms on $\R$. Then, we can see that $(X, \Gamma)$ is weak mixing. However, if we let $N$ be the subgroup of orientation preserving transformations on $\R$, $(X, N)$ is no longer weak mixing and $\Gamma / N \simeq \Z_2$. 
\end{example}

Now, the methods that we sought to use in the previous section are looking less feasible now as in this example, we were not able to pass weak mixing down to a finite index subgroup, even with the condition that $\Gamma / N$ is abelian. 

However, this discussion leads naturally to the following question: we have seen that weak mixing does not pass down to finite index subgroups. Are there other notions of mixing that do pass down to finite index subgroups? In particular, does thick strong mixing pass down to finite index subgroups? 

For further motivation, we can look at the picture in ergodic theory. Here, we have the following result. 

\begin{proposition}[\cite{P}, Corollary 2.2.12]
Let $(X, \mathcal{B}, \mu, \Gamma)$ be a weak mixing, measure preserving system on a probability space where $\Gamma$ is a countable group. If $N \subset \Gamma$ is a finite index subgroup, then $(X, \mathcal{B}, \mu, \Gamma)$ is also weak mixing. 
\end{proposition}

We see by a similar idea as in Example \ref{ex:homeo} that such a proposition does not hold in topological dynamics, even for countable groups: 

\begin{example} Let $X = S^1$ be the unit circle in the complex plane. Let $t$ be an automorphism of $S^1$ with exactly four fixed points at $1, i, -1, - i $ such that $1$ and $-1$ are attracting and $\pm i $ are repelling. Choose $t$ so that it also commutes with conjugation. A graph of such a transformation is plotted below, where if $(x,y)$ is a point on the graph, then $t(e^{i x}) = e^{i y}$: 

\begin{center}
\begin{tikzpicture}[scale = 2]
\node[inner sep=0pt] (russell) at (0,0)
    {\includegraphics[width=.3\textwidth]{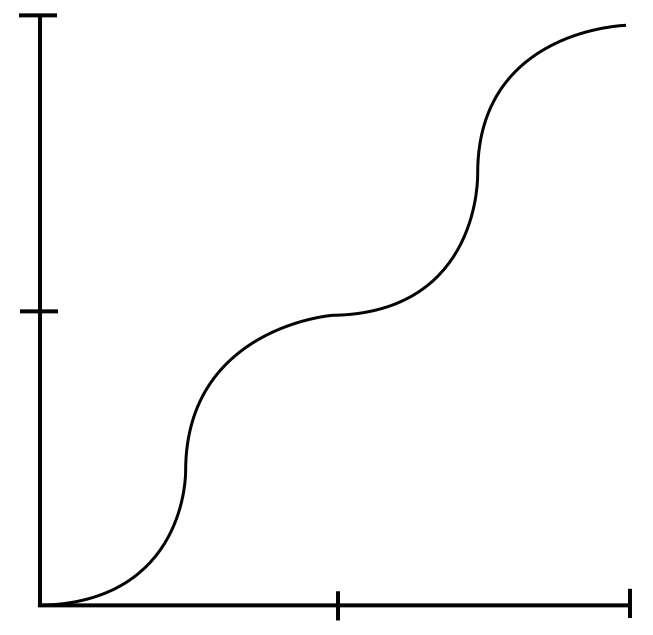}};
\node at (-1.15,-1.15) {$0$};
\node at (-1.25,0) {$\pi$};
\node at (-1.25,1.1) {$2\pi$};
\node at (-0,-1.25) {$\pi$};
\node at (1.1,-1.25) {$2\pi$};

\end{tikzpicture}
\end{center}

Let $c$ be conjugation and $r$ be rotation by $\alpha$, an irrational multiple of $2\pi$. Then, letting $\Gamma = \langle t, c, r : tc = ct,  cr = r^{-1}c, c^2 = 1\rangle$, we have that $(X, \Gamma)$ is weak mixing, but $(X, N)$ is not weak mixing for $N = \langle t, r\rangle$, an index two subgroup of $\Gamma$. 

In this example, we again have that the action of the whole group is weak mixing, but the action of the subgroup is not because these transformations are somehow ``orientation preserving". 
\end{example}

While weak mixing does not pass down to finite index subgroups even in the case of countable groups, we will see that thick strong mixing does pass down, as a result of the following lemma: 

\begin{lemma}[\cite{C}, Lemma 9] If $\Gamma$ is $k \ell$-transitive, then every subgroup of index $\ell$ is $k$-transitive. 
\end{lemma}

From this, the equivalence of $k$-transitivity for all $k$ and thick strong mixing, and that $k$-transitivity on a subgroup implies $k$-transitivity on the whole group, we have the immediate consequence: 

\begin{proposition}
For topological systems $(X, \Gamma)$ where $X$ is second countable, Hausdorff and $\Gamma$ is Hausdorff, second countable, locally compact but not compact, then if $N \subset \Gamma$ is a finite index subgroup
\begin{itemize}
\itemsep0em 
\item[(i)] $(X, \Gamma)$ is $k$-transitive for all $k$ if and only if $(X, N)$ is $k$-transitive for all $k$.
\item[(ii)] As a consequence, if $(X, \Gamma)$ is thick strong mixing if and only if $(X, N)$ is thick strong mixing. 
\end{itemize}
\end{proposition}

This shows that even though weak mixing does not necessarily pass down to finite index subgroups, the stronger property of thick strong mixing on the whole group is equivalent to being thick strong mixing on any finite index subgroup. 

\section{Summary and Concluding Remarks}

We started out by trying to find an analogue of the result in topological dynamics of the result in ergodic theory that under sufficient conditions on the system, a weak mixing transformation is strong mixing on a density one subset of $\N$. As density can sometimes be difficult to define on general groups, we replaced the notion of density one with that of thickness. We called this idea of being strong mixing on a thick subset thick strong mixing and showed that thick strong mixing was equivalent to $k$-transitivity for all $k$. This allowed us to conclude that for abelian groups, weak mixing was equivalent to thick strong mixing. In the $\Z$ case, this mirrored the result from ergodic theory that was our motivation. 

From there, we moved on to find examples to show that weak mixing was not equivalent to thick strong mixing for nonabelian groups. We then asked if there existed a larger class of groups than abelian groups for which weak mixing was equivalent to thick strong mixing. In particular, we looked at virtually abelian and nilpotent groups. While the methods we used to investigate this question were unsuccessful, they led us down a path that allowed us to conclude that weak mixing did not pass down to finite index subgroups as was the case in ergodic theory, but that thick strong mixing did. We also concluded that the property of being thick strong mixing for pairs of nonempty open sets was equivalent to the property of elasticity. 

While we have sketched a picture of the relationship between weak mixing and thick strong mixing in topological dynamics and contrasted it with the picture in ergodic theory, there is still much that is unknown. We know that there exist nonabelian group actions that are both weak mixing and thick strong mixing. We can ask then if there exists a classification of the nonabelian groups for which weak mixing implies thick strong mixing. We can also also ask what happens in the $\Z$ case if we again ask for density one instead of thickness. That is, does weak mixing imply strong mixing on a density one subset of $\Z$? There is still much to be discovered, and we have only scratched the surface of this topic.

\end{document}